\bfseries

\documentclass[12pt,a4]{article}

\usepackage{amsfonts,amsmath,latexsym,amssymb,amsthm}

\usepackage{amsmath}
\usepackage{verbatim}

\usepackage{graphicx}

\usepackage{amsmath}
\newtheorem{theorem}{Theorem}

\newtheorem{corollary}[theorem]{Corollary}

\newtheorem{lemma}[theorem]{Lemma}

\newtheorem{remark}[theorem]{Remark}

\def\d{\mathrm{d}}

\title {Smoothness of scale functions  for spectrally negative
L\'evy processes}

\author{T. Chan\footnote{Actuarial Mathematics and Statistics
School of Mathematical and Computer Sciences
Heriot-Watt University
Edinburgh EH14 4AS
UK. Email: t.chan@ma.hw.ac.uk},  A. E. Kyprianou\footnote{Department of Mathematical Sciences
The University of Bath
Claverton Down
Bath BA2 7AY
UK. Email: a.kyprianou@bath.ac.uk} and M. Savov\footnote{Laboratoire de Probabilites et Modeles Aleatoires,
Universite Paris 6; 4D10, 175, Rue de Chevaleret, 75013  Email: mladensavov@hotmail.com}}

\begin{document}

\maketitle

\begin{abstract}

Scale functions play a central role in the 
fluctuation theory of spectrally negative L\'evy processes and often appear in the context of martingale relations.
These relations are often complicated to establish requiring excursion theory
in favour of It\^o calculus. The reason for the latter is that standard It\^o calculus is only applicable to functions
with a sufficient degree of smoothness and knowledge of the precise degree of smoothness of scale functions is seemingly incomplete.
The aim of this article is to offer new results concerning properties of scale functions in relation to the smoothness of the underlying L\'evy measure. We place particular emphasis on spectrally negative L\'evy processes with a Gaussian component and processes of bounded variation.

An additional motivation is the very intimate relation of scale functions to renewal functions of subordinators. The results obtained for scale functions have direct implications offering new results concerning the smoothness of such renewal functions for which there seems to be very little 
existing literature on this topic.


\end{abstract}


\section{Spectrally negative L\'evy processes and scale functions}
Suppose that $X=\{X_t : t\geq 0\}$ is a spectrally negative L\'evy process with probabilities $%
\{P_x : x\in \mathbb{R}\}$. For convenience we shall write $P$ in place of $P_0$. That is to say a real valued stochastic process whose paths are almost surely right continuous with left limits 
and whose increments are stationary and independent.
Let   $\{\mathcal{F}_t : t\geq 0\}$ be the natural filtration 
satisfying the usual assumptions 
and denote by
 $\psi$ its Laplace exponent 
so that
\begin{equation*}
E(e^{\theta X_t})=e^{t\psi(\theta)}
\end{equation*}
where $E$ denotes expectation with respect to $P$.

It is well known that $\psi$ is finite for all $\theta\geq 0$, is strictly convex on $(0,\infty)$ and satisfies $\psi(0+)=0$ and $\psi(\infty)=\infty$. Further, from the L\'evy-Khintchin formula, it is known that 
\[
\psi(\theta) = a\theta + \frac{1}{2}\sigma^2\theta^2 + \int_{(-\infty,0)} (e^{\theta x} -1 - \theta x \mathbf{1}_{(x> - 1)})\Pi(\d x)
\]
where $a\in\mathbb{R}$, $\sigma^2\geq 0$ and $\Pi$ satisfies $\int_{(-\infty, 0)} (1\wedge x^2)\Pi(\d x)<\infty$ and is called the L\'evy measure. 


Suppose that for each $q\geq 0$, $\Phi(q)$ is the
largest root of the equation $\psi(\theta)=q$ (there are at most two). We
recall from Bertoin (1996, 1997) that for each $q\geq 0$ there exits a
function $W^{(q)}:\mathbb{R}\rightarrow [0,\infty)$, called the $q$-\textit{%
scale function} defined in such a way that $W^{(q)}(x)=0$ for all $x<0$ and
on $[0,\infty)$ its Laplace transform is given by
\begin{equation}
\int_0^\infty e^{-\theta x}W^{(q)}(x)\,\d x = \frac{1}{\psi(\theta)-q}\text{ for 
} \theta > \Phi(q).  \label{LapT}
\end{equation}
For convenience we shall write $W$ in place of $W^{(0)}$
and call this the {\it  scale function} rather than the $0$-scale function.

The importance of $q$-scale functions appears in a number of one and two sided
exit problems for (reflected) spectrally negative L\'evy processes. See for example
Zolotarev (1964), Tak\'acs (1967), Suprun (1976), Emery (1973), Rogers (1990), Bertoin (1996, 1997), Avram et al.
(2004), Doney (2005, 2007), Pistorius (2003, 2004, 2005) and 
Chiu and Yin (2005).
Notably however the $q$-scale function takes its name from the identity
\begin{equation}
E_x (e^{-q\tau_a^+}\mathbf{1}_{(\tau^+_a < \tau^-_0)}) =\frac{W^{(q)}(x)}{%
W^{(q)}(a)}  \label{scale}
\end{equation}
where $\tau^+_a =\inf\{t>0: X_t >a\}$,  $\tau_0^- = \inf\{ t>0 : X_t <0\}$ and $E_x$ is expectation with respect to $P_x$. The latter identity provides an analogue to the situation
for scale functions of diffusions.

Set $q\geq 0$. Under the exponential change of measure
\begin{equation*}
\left. \frac{\d P_x^{\Phi(q)}}{\d P_x}\right|_{\mathcal{F}_{t}}
=e^{\Phi(q)(X_t-x)-q t} 
\end{equation*}
it is well known that
$(X,P^{\Phi(q)})$ is again a spectrally negative L\'evy process
whose Laplace exponent is given by 
\begin{equation}
\psi_{\Phi(q)}(\theta) =\psi(\theta + \Phi(q)) -q  \label{LapPhi}
\end{equation}
for $\theta \geq - \Phi(q)$ and whose L\'evy measure, $\Pi_{\Phi(q)}$, satisfies $$\Pi_{\Phi(q)}(dx) = e^{\Phi(q)x}\Pi(dx)$$ on $(-\infty, 0).$ Note in particular that $\psi_{\Phi(q)\prime}(0+) = \psi'(\Phi(q))$ which is strictly positive when $q>0$ or when $q=0$ and $\psi'(0+)<0$. Using the latter change of measure, we may deduce from (\ref{LapT}) and (\ref{LapPhi}) that
\begin{equation}
W^{(q)}(x)=e^{\Phi(q)x}W_{\Phi(q)}(x) \label{reduce}
\end{equation}
where $W_{\Phi(q)}(x)$ is the scale function for the process $(X, P^{\Phi(q)})$.

There exists an excursion theory argument
given in Bertoin (1996) from which it is known that for any $0<x<\infty$, 
\begin{equation}  \label{integralrep}
W(x)=W(a)\exp\left\{-\int_x^a n(\overline\epsilon > t) \d t \right\} 
\end{equation}
for any arbitrary $a>x$
where $n$ is the excursion measure of the local-time-indexed process of
excursion heights $\{{\overline\epsilon}_t : t\geq 0\}$ of the reflected L\'evy process $\{\sup_{s\leq t}X_s - X_t : t\geq 0\}$. Note that when  $\psi'(0+)>0$ the scale function may also be represented in the form $W(x) = P_x(\inf_{s\leq t}X_s \geq 0)/\psi'(0+)$.
From (\ref{integralrep})  it is immediate that
on $(0,\infty)$ the function $W$ is monotone and almost everywhere differentiable with left and right derivatives given by
\begin{equation}  \label{W-dash}
W^{\prime}(x-)=n(\overline\epsilon \geq x)W(x) \text{ and }
W^{\prime}(x+) =n(\overline\epsilon > x)W(x)
\end{equation}
so that $W$ is continuously differentiable  on $(0,\infty)$ if and only if $n({\overline\epsilon} =t)=0$
for all $t>0$ in which case $W^{\prime}(x)  = n(\overline{\epsilon}\geq x)W(x)$.
In Lambert (2000) it is shown that the latter is the case if for example the process $X$ has
paths of unbounded variation (and in particular if it possesses a Gaussian
component) or if $X$ has paths of bounded variation and the L\'evy measure
is absolutely continuous with respect to Lebesgue measure. More recently, it has been shown in Doney (2007) and Kyprianou et al. (2008) that, in the case of bounded variation paths, $W$ is continuously differentiable if and only if $\Pi$ has no atoms.

The principle objective of this paper is to investigate further the smoothness properties of scale functions. In particular we are interested in providing conditions on the L\'evy measure $\Pi$ and Gaussian coefficient $\sigma$ such that, for all $q\geq 0$, the restriction of $W^{(q)}$ to $(0,\infty)$ belongs to $C^k(0,\infty)$ for $k=2,3,\cdots$.   (Henceforth we shall write the latter with the slight abuse of notation $W^{(q)}\in C^k(0,\infty)$).

\section{Motivation and main results}\label{motivation}

Before moving to the main results, let us motivate further the reason for studying the smoothness of scale functions. From (\ref{scale}) it is straightforward to deduce by applying the Strong Markov Property that for $a>0$ 
\[
\{e^{-qt}W^{(q)}(X_t):   t<\tau^+_a \wedge \tau^-_0\}
\]
is a martingale. In the spirit of the theory of stochastic representation associated with one dimensional diffusions, this fact may in principle be used to solve certain boundary value problems by affirming that $W^{(q)}$ solves the integro-differential equation 
\[
(\Gamma -q)W^{(q)}(x)=0 \text{ on }(0,a)
\]
where $\Gamma$ is the infinitesimal generator of $X$.

The latter equation would follow by an application of It\^o's formula to the aforementioned martingale providing that one can first assert sufficient smoothness of $W^{(q)}$. In general, a comfortable sufficient condition to do this would be that $W^{(q)}\in C^2(0,\infty)$, although this is not strictly necessary. For example, when $X$ is of bounded variation, knowing that $W^{(q)}\in C^1(0,\infty)$ would suffice.


In general it is clear from (\ref{LapT}) that imposing conditions on the L\'evy triple $(a,\sigma,\Pi)$, in particular the quantities $\sigma$ and $\Pi$, is one way to force a required degree of smoothness on $W^{(q)}$. It is also possible to get some a priori feeling for what should be expected by way of results by looking briefly at the intimate connection between scale functions when $\psi'(0+)\geq  0 $ and  $q=0$ and renewal functions. 

Suppose that $X$ is such that $\psi'(0+)\geq 0$, then thanks to the Wiener-Hopf factorisation we may write $\psi(\theta) = \theta \phi(\theta)$ where $\phi$ is the Laplace exponent of the descedning ladder height subordinator. (See for example Chapter VI of Bertoin (1996) for definitions of some of these terms). Note that in this exposition, we understand a subordinator in the broader sense of a L\'evy process with non-decreasing paths which is possibly killed at an independent and exponentially distributed time. A simple integration by parts in (\ref{LapT}) shows that 
\begin{equation}
 \int_{[0,\infty)}e^{-\theta x}W({\rm d} x) = \frac{1}{\phi(\theta)}
 \label{remark-around}
\end{equation}
which in turn uniquely identifies $W({\rm d}x)$ as  the renewal measure associated with the descending ladder height subordinator. Recall that if $H=\{H_t: t\geq 0\}$ is the subordinator associated with the Laplace exponent $\phi$ then 
a very first view on smoothness properties could easily incorporate known facts concerning smoothness properties of renewal measures. 

For example, Hubalek and Kyprianou (2008) formalise several facts which are implicit in the Wiener-Hopf factorisation. Specifically it is shown that given any Laplace exponent of a subordinator, $\phi$, there exists a spectrally negative L\'evy process with Lapalce exponent $\psi(\theta) = \theta\phi(\theta)$ if and only if  the L\'evy measure associated to $\phi$ is absolutely continuous with non-increasing density. In that case the aforementioned density is necessarily equal to $$\overline{\Pi}(x):=\Pi(-\infty,-x).$$ A sub-class of such choices for $\phi$ are the so called complete Bernstein functions. That is to say, Laplace exponents of subordinators whose L\'evy measure is absolutely continuous with completely monotone density.
A convenience of this class of $\phi$ is that, from Theorem 2.3 of Rao et al. (2005) and Theorem 2.1 and Remark 2.2 of Song and Vondra\v{c}ek (2005), the potential density associated to $\phi$, that is to say $W'$, is completely monotone and hence  belongs to $C^\infty(0,\infty)$.
This way, one easily reaches the conclusion  that  whenever $\psi'(0+)\geq 0$ and $\overline{\Pi}(x)$ is completely monotone, then $W'$ is completely monotone. (See Kyprianou and Rivero (2008) for further results in this direction).

The above arguments show that a strong smoothness condition on the L\'evy measure $\Pi$, namely complete monotonicity of $\overline{\Pi}$, yields a strong degree of smoothness on the scale functions; at least for some particular parameter regimes. One would hope then that a weaker smoothness condition on the L\'evy measure $\Pi$, say $\overline{\Pi}\in C^n(0,\infty)$ for $n=1,2,...$, might serve as a suitable condition in order to induce a similar degree of smoothness for the associated scale functions.
Another suspicion one might have given the theory of scale functions for diffusions is that, irrespective of the jump structure,  the presence of a Gaussian component is always enough to guarantee that the scale functions are $C^2(0,\infty)$.

\bigskip

 Let us now turn to our main results which do indeed reflect these intuitions. We deal first with the case of when a Gaussian component is present.




\begin{theorem}\label{I}
Suppose that $X$ has a Gaussian component. For each fixed $q\geq
0$ the function $W^{(q)}$ belongs to the class $C^{2}(0,\infty)$.
\end{theorem}

 \begin{theorem}\label{II}
 Suppose that $X$ has a Gaussian component and its Blumenthal-Getoor index belongs to  $[0,2)$, that is to say
 \[
 \inf\{\beta\geq0: \int_{|x|<1} |x|^\beta \Pi({\rm d}x)<\infty\}\in[0,2).
 \]
 Then for each $q\geq 0$ and $n=0,1,2,...$, $%
W^{(q)}\in C^{n+3}(0,\infty)$ if and only if $\overline{\Pi}\in C^{n}(0,\infty)$.
\end{theorem}

\noindent Next we have a result concerning the case of bounded variation paths.

 \begin{theorem}\label{III}
 Suppose that $X$ has paths of bounded variation and its  tail has a derivative $\pi(x)$, such that   $\pi(x)\leq Cx^{-1-\alpha}$ in the neighbourhood of the origin, for some $\alpha<1$ and $C>0$. Then for each $q\geq 0$ and $n=1,2,...$, $%
W^{(q)}\in C^{n+1}(0,\infty)$ if and only if $\overline{\Pi}\in C^{n}(0,\infty)$.\end{theorem}

\begin{remark}\rm
Note that if $\overline\Pi(0)<\infty$ (that is to say $X$ has a compound Poisson jump structure), then the Blumenthal-Getoor index is zero. As a consequence Theorem \ref{II} implies, without further restriction on $\Pi$, that, if in addition a Gaussian component is present in $X$, then, for $n=0,1,2,\cdots$,  $W^{(q)}\in C^{n+3}(0,\infty)$ if and only if $\overline{\Pi}\in C^{n}(0,\infty)$.
\end{remark}

\begin{remark}\rm
In Kyprianou et al. (2008) related results concerning the case where the underlying L\'evy process has a Gaussian component exist. In particular, it was shown that if $\overline{\Pi}$ is log-convex and absolutely continuous with respect to Lebesgue measure, then for all $q\geq 0$, $W^{(q)\prime}$ is ultimately convex and $W^{(q)}\in C^2(0,\infty)$. Also as an elaboration on some of the discussion in Section \ref{motivation}, it was shown in Loeffen (2008) that if $\overline{\Pi}$ has a completely monotone density, then for all $q\geq 0$, $W^{(q)\prime}$ is convex and $W^{(q)}\in C^\infty (0,\infty)$.
\end{remark}

\noindent The various methods of proof we shall appeal to in order to establish the above three theorems reveals that the case that $X$ has paths of unbounded variation but no Gaussian component is a much more difficult case to handle and unfortunately we are not able to offer any concrete statements in this regime. 

\bigskip

We conclude this section with a brief summary of the remainder of the paper. In the next section we shall look at some associated results  which concern smoothness properties of a certain family of renewal measures. These results will form the basis of one of two key techniques used in the proofs. Moreover, this analysis will also lead to a new result, extending a classical result of Kesten (1969), concerning smoothness properties of renewal measures of subordinators with drift.  

We then turn to the proofs of our results on renewal equations and finally use them, together with other probabilistic techniques, to prove the results on scale functions.




\section{Renewal equations}
Henceforth the convolution of two given functions, $f$ and $g$ mapping $[0,\infty)$ to $\mathbb{R}$, will be defined as 
\[f*g(x)=\int_{0}^{x}f(x-y){\rm d}g(y).\]
In all of the subsequent analysis it will always be the case $g$ is absolutely continuous with respect to Lebesgue measure  which suffices for the correctness of the above definition. When appropriate, we shall also understand $g^{*n}$ to be the $n$-fold convolution of $g$ on $(0,\infty)$ where $n=0,1,\cdots$. In particular, $g^{*0}({\rm d}x) = \delta_0({\rm d}x)$, $g^{*1} = g$ and for $n=2,3,\cdots$,
\[
g^{*n}(x)  = \int_{0}^{x}g^{*(n-1)}(x-y)g'(y){\rm d}y
\] 
on $(0,\infty)$. 

We also note that in the following theorems, the issue of uniqueness of solutions to the renewal equation is already well understood. However we include the statement and proof of uniqueness for completeness.

\begin{theorem}\label{Sub}
Let $g$ be a negative, decreasing function on $(0,\infty)$ with $g(0)=0$ and $|g'(x)|$ is continuous and decreasing such that $g'(\infty)=0$. Moreover assume that on every interval $(0,a)$ there exists a constant $C(a)$ such that $|g'(x)|\leq C(a)/x$. Then the solution of the renewal equation
\begin{equation}\label{renewal equation}
f(x)=1+f*g(x)
\end{equation}
is unique in the class of functions which are bounded on bounded intervals and has the usual form \[f(x)=\sum_{n\geq 0}g^{*n}(x).
\] 
Moreover, the first derivative $f'$ exists, $f'(x)=\sum_{n\geq0}g^{*n\prime}(x)$ and $f'$ is continuous.  
\end{theorem}
\begin{theorem}\label{Renewal}
Let $g\in C^{2}(0,\infty)$ be a function satisfying $g(0)=0$ and $|g'(x)|\leq Cx^{-\alpha}$ in a neighbourhood of   $0$, for some $C>0$ and $\alpha<1$. Then the solution of the renewal equation \eqref{renewal equation} is unique in the class of functions which are bounded on bounded intervals and has the form $f(x)=\sum_{n\geq0}{g^{*n}}(x)$ on any interval $[0,a)$. If in addition $|g''(x)|\leq C(a)/x^{\alpha+1}$ on any interval $(0,a)$, then, for any $k=0,1,2,\cdots$, $f \in C^{k+2}(0,\infty)$ if and only if $g\in C^{k+2}(0,\infty)$.
\end{theorem}
\begin{remark}\label{ThmRenRem}\rm
Note that according to the conditions of the previous theorem, $g \in C^{2}(0,\infty)$ and therefore $f \in C^{2}(0,\infty)$.
\end{remark}

Although the above two theorems will be used to address the issue of smoothness properties of scale functions, we may also deduce some new results concerning  renewal functions of subordinators. 
Indeed this is the purpose of the next corollary which generalises a classical result of Kesten (1969). The latter says that whenever  $U({\rm d}x)$ is the renewal measure of a subordinator with  drift coefficient  $\delta>0$ then $U$ is absolutely continuous with respect to Lebesgue measure with a  density, $u(x)$ which has a continuous version (which in turn is $\delta^{-1}$ multiplied by the probability that the underlying subordinator crosses the level $x$ by creeping) for $x>0$.


\begin{corollary}\label{IV}
Suppose that $U$ is the renewal measure of a subordinator with positive drift $\delta$ and a L\'evy measure $\mu$ which is also assumed to have no atoms. Then the renewal density has the form \begin{equation}\label{Den1}u(x)=\sum_{n\geq0}\frac{\eta^{*n}(x)}{\delta^{n+1}},\end{equation} where $\eta^{*(n+1)}(x)=-\int_{0}^{x}\eta^{*n}(x-y)\overline{\mu}(y){\rm d}y$, $\eta(x)=-\int_{0}^{x}\overline{\mu}(y){\rm d}y$ and $\overline{\mu}(x) = \mu(x,\infty)$ and it is continuously differentiable with
\begin{equation}\label{Den1}u'(x)=\sum_{n\geq0}\frac{\eta^{*n\prime}(x)}{\delta^{n+1}}.\end{equation} 
\end{corollary}

\begin{proof} Let us temporarily assume that the underlying subordinator is not subject to killing.
Then the probability of this subordinator crossing a level $x$ is $1$ and this can happen either by creeping or jumping over this level. With the help of Kesten (1969)  this can otherwise be written as
\begin{equation}\label{Crossing} 
\delta u(x)+\int_{0}^{x}u(x-y)\overline{\mu}(y){\rm d}y=1
\end{equation}
(see for example Chapter 3 of Bertoin (1996)). Now letting $f(x) = \delta u(x)$ and $g(x) = -\delta^{-1} \int_0^x \overline{\mu}(y){\rm d}y$ and noting that the necessary condition on subordinator L\'evy measures, $\int_{(0,1)}x\mu({\rm d}x)<\infty$, implies that $\lim_{x\downarrow 0}x\overline{\mu}(x)=0$, we see that the conditions of Theorem \ref{Sub} are satisfied and the statement of the corollary follows.

  Now we turn to the case when the underlying subordinator is killed. In that  case, suppose that it has the same law as an unkilled subordinator, say $Y$ with Laplace exponent $\eta(\theta) = - \log{E}(e^{-\theta Y_1})$ for $\theta\geq 0$,  killed at rate $q>0$. The result of Kesten (1969) tells us that 
\[
U({\rm d}x) = \frac{1}{\delta} {P}(e^{- qT_x^+}; Y_{T^+_x} = x){\rm d}x
\]
where $T^+_x = \inf\{t>0: Y_t = x\}$ (see for example Exercise 5.5 in Kyprianou (2006)). We may thus write 
\[
U({\rm d}x) =  \frac{1}{\delta}e^{\eta^{-1}(q)x} {P}(e^{-\eta^{-1}(q)x- qT_x^+}; Y_{T^+_x} = x){\rm d}x
=\frac{1}{\delta}e^{\eta^{-1}(q)x} U_{\eta^{-1}(q)}({\rm d}x)
\]
where $U_{\eta^{-1}(q)}({\rm d}x)$ is the renewal measure of the subordinator $Y$ when seen under the exponential change of measure associated with the martingale $\{e^{-\eta^{-1}(q)Y_t - qt}:t\geq 0 \}$ and $\eta^{-1}$ is the right inverse of $\eta$. Note that the Laplace exponent of process $Y$ under the aforementioned change of measure is given by $\eta(\theta+ \eta^{-1}(q))-q$ for $\theta\geq 0$ and it is straightforward to deduce that there is no killing term and the L\'evy  measure is given by $e^{-\eta^{-1}(q)x}\mu({\rm d}x)$.  Since $U_{\eta^{-1}(q)}({\rm d}x)$ is the potential measure of an unkilled subordinator and  since the behaviour at the origin of the measure $e^{-\eta^{-1}(q)x}\mu({\rm d}x)$ is identical to that of $\mu$ with regard to the role it plays through the function $g$, the result for the case of killed subordinators follows from the first part of the proof. \end{proof}

\section{Renewal equation proofs}
In this section we shall prove Theorems \ref{Sub} and \ref{Renewal}. Before doing so, we first need to establish the following auxiliary lemma.

\begin{lemma}\label{Properties}
Let g be a negative, decreasing function on $(0,\infty)$ with $g(0)=0$
and let $|g'(x)|$ be decreasing. Then $|g^{*n}(x)|=(-1)^{n}g^{*n}(x)$. Moreover $|g^{*n\prime}(x)|=(-1)^{n}g^{*n\prime}(x)$ and 
\begin{equation}\label{derivatives}
g^{*(n+1)\prime}(x)=\int_{0}^{x}g'(x-y)g^{*n\prime}(y){\rm d}y=\int_{0}^{x}g^{*n\prime}(x-y)g'(y){\rm d}y.
\end{equation}
In conclusion $g^{*n}(x)$ is increasing for $n$ even and decreasing otherwise.
\end{lemma}

\begin{proof}
The first statement is obvious from the definition of convolution $$g^{*(n+1)}(x):=\int_{0}^{x}g^{*n}(x-y)g'(y){\rm d}y.$$ Next we prove  $|g^{*n\prime}(x)|=(-1)^{n}g^{*n\prime}(x)$ and  \eqref{derivatives}. We achieve this by an inductive argument. Note that $g'(x)$ is absolutely integrable and write, for $h>0$,
\begin{eqnarray*}
|g^{*2}(x+h)-g^{*2}(x)| &\leq& \int_{0}^{x}\big|g(x+h-y)-g(x-y)\big||g'(y)|{\rm d}y\\
&&\hspace{1cm}+|g(h)||g(x+h)-g(x)|.
\end{eqnarray*}
The fact that $|g'(x)|$ is decreasing implies that $\big|g(x+h-y)-g(x-y)\big|\leq h|g'(x-y)|=-h g'(x-y)$. Finally from the existence of $\int_{0}^{x}g'(x-y)g'(y){\rm d}y$, the nonpositivity of $g'(x)$ and the dominated convergence theorem it can be deduced that
\[|g^{*2\prime}(x+)|=g^{*2\prime}(x+)=\lim_{h\rightarrow0+}\frac{|g^{*2}(x+h)-g^{*2}(x)|}{h}=\int_{0}^{x}g'(x-y)g'(y){\rm d}y.
\] 
This serves as a basis for an induction hypothesis, which states that $|g^{*n\prime}(x+)|=(-1)^{n}g^{*n\prime}(x+)$ and 
\[g^{*n\prime}(x+)=\int_{0}^{x}g'(x-y)g^{*(n-1)\prime}(y){\rm d}y=\int_{0}^{x}g^{*(n-1)\prime}(x-y)g'(y){\rm d}y.\]
To show these statements for $n+1$, we use the induction hypothesis, the preceding part of the proof and 
\begin{align*}
&|g^{*(n+1)}(x+h)-g^{*(n+1)}(x)|\leq\\& \int_{0}^{x}|g(x+h-y)-g(x-y)||g^{*n\prime}(y)|{\rm d}y+|g(h)||g^{*n}(x+h)-g^{*n}(x)|,
\end{align*}
to justify the use of the dominated convergence theorem and hence deduce the statement for $n+1$.

We conclude the proof by noting that the same arguments work for $h<0$ and we can show that $g^{*n\prime}(x+)=g^{*n\prime}(x-)=\int_{0}^{x}g'(x-y)g^{*(n-1)\prime}(y){\rm d}y$. The last statement of the lemma is  an  obvious consequence of the other assertions.
\end{proof}

\bigskip

\begin{proof}[Proof of Theorem \ref{Sub}]
First we discuss the properties of $\phi(x):=\sum_{n\geq 0}g^{*n}(x)$ and in the end we show that this is the only solution of equation \eqref{renewal equation}. We start by showing that on every interval $(0,a)$ the series defining $\phi(x)$ is uniformly convergent and therefore $\phi(x)$ is continuous on $(0,a)$.

It is clear from the conditions of the theorem, i.e. $g(0)=0$, that there is an interval $(0,b]$ such that $|g(x)|<1$. Then we get
\begin{equation}\label{Estimate}
|g^{*n}(x)|=|\int_{0}^{x}g^{*(n-1)}(x-y)g'(y){\rm d}y|\leq|g^{*(n-1)}(x)||g(x)|\leq |g(x)|^{n},
\end{equation}
which follows directly from the fact that $|g^{*k}(x)|$ is increasing in $x$, for each $k\geq1$ (cf. Lemma \ref{Properties}) together with an obvious iteration of the first inequality. Therefore the series $\sum_{n\geq 0}g^{*n}(x)$ is uniformly convergent on $(0,b]$. Next, we show how this can be extended to the interval $(0,2b]$. 

We know from Lemma \ref{Properties} that $|g^{*n}(x)|$ is nondecreasing and therefore it is sufficient to show that $g^{*n}(2b)\leq Cn\gamma^{n}$, for each $n$ and some $\gamma<1$. We then set up an induction hypothesis that $|g^{*n}(2b)|\leq n|g(b)|^{n-1}|g(2b)|$, which clearly holds for $n=1$.
Note also  that
\begin{eqnarray*}
|g^{*(n+1)}(2b)|&\leq& \int_{0}^{b}|g^{*n}(2b-x)||g'(x)|dx +\int_{b}^{2b}|g^{*n}(2b-x)||g'(x)|dx\\
&\leq&
|g^{*n}(2b)||g(b)|+|g^{*n}(b)||g(2b)|\\
&\leq& n|g(b)|^{n}|g(2b)|+|g(b)|^{n}|g(2b)|\\
&=&(n+1)|g(b)|^{n}|g(2b)|,
\end{eqnarray*} 
where we have used the properties stated in Lemma \ref{Properties} and \eqref{Estimate}. Taking into account the fact that $|g(b)|<1$, we deduce that 
the series $\sum_{n\geq0}g^{*n}(x)<\infty$ is uniformly convergent on $(0,2b]$.

Next, we show how the uniform convergence can be extended to the interval $(0,4b]$. Since $\sum_{n\geq0}g^{*n}(x)<\infty$ is uniformly convergent on $(0,2b]$, we can find $k$ such that $g^{*k}(2b)<1$ and apply the previous arguments to the function $g^{*k}$ to deduce that $\sum_{l\geq 1}g^{*(l k)}(x)$ is uniformly convergent on $(0,4b]$. For a series of the type $\sum_{l\geq 0}g^{*( l k+j)}(x)$, for $0\leq j\leq k-1$, we easily get uniform convergence using the trivial estimate $|g^{*(l k+j)}(x)|\leq |g^{*(lk)}(x)||g(x)|^{j}$, the fact that $|g(x)|^{j}\leq |g(4b)|^{j}$  and the uniform convergence of $\sum_{l\geq 0}g^{*(l k)}(x)$. Finally note that $\phi(x)  = \sum_{j=1}^{k-1} \sum_{l\geq 0}g^{*(lk+j)}(x)$. This process can be continued {\it ad infinitum} and therefore it implies that $\phi(x)$ is well defined on $(0,\infty)$.

Next we wish to show that $\phi'(x)=\sum_{n\geq 1}g^{*n\prime}(x)$. We pick an interval $(\epsilon,b]$, where $b$ is chosen in a way that $|g(b)|=q<1/4$. Using \eqref{derivatives} in Lemma \ref{Properties} and the fact that $|g'|=-g'$ is nonincreasing, we obtain, for each $x>0$, 
\[|g^{*2\prime}(x)|\leq 2|g'\big(\frac{x}{2}\big)||g\big(\frac{x}{2}\big)|.\] This allows us to set up the following induction hypothesis, for each $x>0$,
\[|g^{*n\prime}(x)|\leq n |g\big(\frac{x}{2}\big)|^{n-1}|g'\big(\frac{x}{2^{n-1}}\big)|.\]
Then we check using \eqref{derivatives}, the induction hypothesis and \eqref{Estimate}, that
\begin{eqnarray}
 \big|g^{*(n+1)\prime}(x)\big|&\leq& \int_{0}^{\frac{x}{2}}|g^{*n\prime}(x-y)||g'(y)|{\rm d}y+\int_{\frac{x}{2}}^{x}|g^{*n\prime}(x-y)||g'(y)|{\rm d}y\notag \\
 &\leq &
\sup_{\frac{x}{2}\leq s\leq x}\big|g^{*n\prime}(s)\big|\big|g\big(\frac{x}{2}\big)\big|+\big|g'\big(\frac{x}{2}\big)\big|\big|g^{*n}\big(\frac{x}{2}\big)\big|\notag\\
&\leq &
n \big|g\big(\frac{x}{2}\big)\big|^{n-1}\big|g'\big(\frac{x}{2^{n}}\big)\big|\big|g\big(\frac{x}{2}\big)\big|+\big|g\big(\frac{x}{2}\big)\big|^{n}\big|g'\big(\frac{x}{2^{n}}\big)\big|\notag\\
&=&(n+1)\big|g\big(\frac{x}{2}\big)\big|^{n}\big|g'\big(\frac{x}{2^{n}}\big)\big|.
\label{EEstimate1}
\end{eqnarray}
Finally we recall that according to the conditions of Theorem \ref{Sub} and the choice of $b$, $|g'(x)|\leq C(b)/x$ and $|g(x)|=q<1/4$, for each $x\in (\epsilon,b)$. Therefore
\[\sum_{n\geq 1}|g^{*n\prime}(x)|\leq \frac{C(b)}{\epsilon}\sum_{n\geq 1}(n+1)2^{n}q^{n}<\infty.\]
 Thus, we conclude that, for each $\epsilon>0$, the series $\sum_{n\geq 1}g^{*n\prime}(x)$ is uniformly convergent on $(\epsilon,b)$ and hence by the dominated convergence theorem $\phi'(x)$ exists and equals $\sum_{n\geq 1}g^{*n\prime}(x)$. Since $\epsilon$ is arbitrary, the latter conclusion holds for $x\in(0,b]$.

 We aim next at extending this identity to the interval $(0,2b]$. We follow a similar argument to the one applied to the function $\phi(x)$ itself. For this purpose we fix an interval $(\epsilon,2b]$ and choose $k$ so large that 
\[q=\max_{0\leq j<k}\sup_{s\in (0,b]}|g^{*(k+j)}(s)|<\frac{1}{4}.\]
This can be done due to the uniform convergence of the series $\sum_{n\geq 0}g^{*n}(x)$ on $(0,b]$. Then the proof follows a well established pattern from above. It can directly be estimated using \eqref{EEstimate1} that, for each $x>0$,
\begin{eqnarray*}
\big|g^{*(2k+j)\prime}(x)\big| 
&\leq& \int_{0}^{\frac{x}{2}}\big|g^{*k\prime}(x-y)\big|\big|g^{*(k+j)\prime}(y)\big|{\rm d}y\\
&&+\int_{\frac{x}{2}}^{x}\big|g^{*k\prime}(x-y)\big|\big|g^{*(k+j)\prime}(y)\big|{\rm d}y\\
&\leq&
\sup_{\frac{x}{2}\leq s\leq x}\big|g^{*k\prime}(s)\big|\big|g^{*(k+j)}\big(\frac{x}{2}\big)\big|+\sup_{\frac{x}{2}\leq s\leq x}\big|g^{*(k+j)'}(s)\big|\big|g^{*k}\big(\frac{x}{2}\big)\big|\\
&\leq& k\big|g'\big(\frac{x}{2^{k}}\big)\big|\big|g\big(\frac{x}{2}\big)\big|^{k-1}\big|g^{*(k+j)}\big(\frac{x}{2}\big)\big|\\
&&+(k+j)\big|g'\big(\frac{x}{2^{k+j}}\big)\big|\big|g\big(\frac{x}{2}\big)\big|^{k+j-1}\big|g^{*k}\big(\frac{x}{2}\big)\big|\\
&\leq& 2(k+j)\big|g'\big(\frac{x}{2^{k+j}}\big)\big|
\Big(
\max_{0\leq j<k}
\big|
g^{*(k+j)}(\frac{x}{2})
\big|
\Big)
\Big(\max\big\{\big|g\big(\frac{x}{2}\big)\big|^{k+j-1},\big|g\big(\frac{x}{2}\big)\big|^{k-1}\big\}\Big).
\end{eqnarray*}
Therefore, for each $x\in (0,2b]$, we get
\[|g^{*(2k+j)\prime}(x)|\leq 2(k+j)Q(x)q\big|g'\big(\frac{x}{2^{k+j}}\big)\big|,\]
where $Q(x)=Q(k,j,x)=\max\{|g(\frac{x}{2})|^{k+j-1},0\leq j<k\}$ is a non-decreasing function in $x$.

This allows us to set up the following induction hypothesis. For each $0\leq j\leq k-1$ and $x\in (0,2b]$,
\[\big|g^{*(nk+j)\prime}(x)\big|\leq n(k+j)Q(x)q^{n-1}\big|g'\big(\frac{x}{2^{k+j+n}}\big)\big|.\]
Then we write using the latter
\begin{eqnarray*}
\big|g^{*((n+1)k+j)\prime}(x)\big|&=&\big|\int_{0}^{x}g^{*(nk+j)\prime}(x-y)g^{*k\prime}(y){\rm d}y\big|\\
&\leq&
\int_{0}^{\frac{x}{2}}\big|g^{*(nk+j)\prime}(x-y)\big|\big|g^{*k\prime}(y)\big|{\rm d}y\\
&&+\int_{\frac{x}{2}}^{x}\big|g^{*(nk+j)\prime}(x-y)\big|\big|g^{*k\prime}(y)\big|{\rm d}y\\
&\leq&
\sup_{\frac{x}{2}\leq s\leq x}\big|g^{*(nk+j)\prime}(s)\big|\big|g^{*k}\big(\frac{x}{2}\big)\big|+\sup_{\frac{x}{2}\leq s\leq x}\big|g^{*k\prime}(s)\big|\big|g^{*(nk+j)}\big(\frac{x}{2}\big)\big|\\
&\leq& n(k+j)Q(x)q^{n-1}\big|g'\big(\frac{x}{2^{k+j+n+1}}\big)\big|q+k\big|g\big(\frac{x}{2}\big)\big|^{k-1}\big|g'\big(\frac{x}{2^{k}}\big)\big|q^{n}\\ 
&\leq&(n+1)(k+j)Q(x) q^{n}\big|g'\big(\frac{x}{2^{k+j+n+1}}\big)\big|,
\end{eqnarray*}
where for the second inequality we have used the fact $|g^{*n\prime}(x)|=(-1)^{n+1}g^{*n\prime}(x)$ (cf. Lemma \ref{Properties}) to do the integration, and in the third inequality \eqref{EEstimate1}, the induction hypothesis and $|g^{*n k+j}(\frac{x}{2})|\leq |g^{*k}(\frac{x}{2})|^{n-1}|g^{*(k+j)}(\frac{x}{2})|\leq q^{n}$, for $x\leq 2b$, have been invoked. Finally, we recall that there is  a $C(2b)$ such that $|g'(\frac{x}{2^{k+j+n}})|\leq C(2b)2^{k+j+n}/x$ holds on $x\in(0,2b]$ to deduce that 
\[\sum_{n\geq 1}|g^{*(kn+j)\prime}(x)|\leq \frac{1}{\epsilon}\sum_{n\geq 1}(n+1)Q(2b)C(2b) q^{n}2^{k+j+n}<\infty,\]
for each $x\in (\epsilon,2b]$  (taking account of the fact that  $|g'(0+)|$ can be infinity) and $j\leq k-1$. Hence we conclude that $\sum_{n\geq 1}g^{*n\prime}(x) = \sum_{j=0}^{k-1}\sum_{n\geq 0}g^{*(kn+j)\prime}(x)$ is uniformly convergent on $(\epsilon,2b]$ and by the dominated convergence theorem it follows that $\phi'(x)$ exists and equals $\sum_{n\geq 1}g^{*n\prime}(x)$. Since this procedure can be repeated forever and $\epsilon$ is arbitrary, we are in position to conclude that $\phi'(x)=\sum_{n\geq 1}g^{*n\prime}(x)$ on $(0,\infty)$.

We conclude our proof by addressing  the uniqueness of the solution of \eqref{Renewal} is proved. First note that it is straightforward to deduce that $\phi$ solves \eqref{Renewal}. Indeed, thanks to Fubini's theorem we may write 
\[
 \phi(x) = 1 + \sum_{n\geq 0}\int_0^x g^{*n}(x-y)g'(y)dy 
= \phi(x) = 1 + \int_0^x \phi(x-y)g'(y)dy,
\]
noting in particular that the convolution on the right hand side is well defined.
Next assume that there are two bounded solutions on $[0,a]$. Denote them by $v_{1}$ and $v_{2}$. Then $v=v_{1}-v_{2}=v*g$ and in general $v=v*g^{*n}$, for each $n$. The simple estimate
\[|v(x)|\leq \sup_{0\leq s\leq x}|v(s)||g^{*n}(x)|\]
implies that $v(x)\equiv 0$ since $\lim_{n\to\infty}|g^{*n}(x)|=0$, for each $x$ on account of the convergence of the sum which defines $\phi(x)$. Therefore we conclude that $\phi(x)=f(x)$ is the unique solution of \eqref{Renewal}, which is bounded on any interval $[0,a]$. 
\end{proof}

\bigskip

\begin{proof}[Proof of Theorem \ref{Renewal}]
We start by showing the uniform and absolute convergence of $\sum_{n\geq 0}g^{*n}(x)$. First note that  $|g'(x)|\leq Cx^{-\alpha}$, for some $\alpha\in (0,1)$, and the functions $g^{*n}$ are well defined. Also the following estimate is immediate by a simple recursion \begin{equation}\label{Estimate1}
|g^{*n\prime}(x)|\leq C_{n}(a)x^{n-1-n\alpha} 
\end{equation} 
for each $n$ and each interval $(0,a)$.
Indeed fix an interval $(0,a)$ and assume that $x\in(0,a)$. We check using the integrability of $|g'(x)|$ about $0$ and $|g'(x)|\leq Cx^{-\alpha}$ that
\[g^{*2\prime}(x)=\frac{\rm d}{{\rm d}x}\int_{0}^{x}g(x-y)g'(y){\rm d}y=\int_{0}^{x}g'(x-y)g'(y){\rm d}y\] and
using the property of $g'$ we get
\[|g^{*2\prime}(x)|\leq C^{2}\int_{0}^{x}\frac{1}{(x-y)^{\alpha}y^{\alpha}}{\rm d}y=C_{2}(a)x^{1-2\alpha}\]
 Therefore we conclude that $g^{*2\prime}$ is continuous and integrable around $0$. This allows us by simple recursion and induction to deduce that in general $g^{*n\prime}(x)=\int_{0}^{x}g^{*(n-1)\prime}(x-y)g'(y){\rm d}y$ is continuous and \eqref{Estimate1} follows.
 
Taking into account that $\alpha<1$ and \eqref{Estimate1}, we see that eventually there is $k$ such that, for each $n\geq k$, $\lim_{x\to 0} g^{*n\prime}(x)=0$. This fact and \eqref{Estimate}, imply that
\[|g^{*(k+l)}(x)|=|\int_{0}^{x}g^{*l}(x-y)g^{*k\prime}(y){\rm d}y|\leq \sup_{y\leq x}|g^{*k\prime}(y)||g(x)|^{l}x\leq D(a)S(a)x,\]
for each $x\in(0,a)$, any $l=0,..,k-1$ with $D(a)=\sup_{y\leq a}|g^{*k\prime}(y)|$ and $S(a)=\max\{l\leq k-1: |g(a)|^{l}\}$. Finally we deduce by induction that, for each $n\geq 1$,
\[|g^{*(nk+l)}(x)|=|\int_{0}^{x}g^{*((n-1)k+l)}(x-y)g^{*k\prime}(y){\rm d}y|\leq S(a)D^{n}(a)\frac{x^{n}}{n!}.\] 
As in the proof of Theorem \ref{Sub}, This is enough to show that $\sum_{n\geq0}g^{*n}(x)$ is uniformly and absolutely convergent on any interval $(0,a)$ and moreover 
\begin{equation}\label{f}
f(x)=\sum_{n\geq0}g^{*n}(x).
\end{equation} 
In addition, a similar argument to the proof of Theorem \ref{Sub} shows that this is also the unique solution in the class of solutions which are bounded on bounded intervals.

Next we move on to the higher order smoothness properties. To this end, assume that $g\in C^{l+1}(0,\infty)$. From \eqref{f} and the uniform and absolute convergence of the series we conclude that if $g\in C(0,\infty)$ then $f\in C(0,\infty)$ and we proceed by induction to deal with higher derivatives.  The induction hyposthesis we shall use is that $g\in C^{l}(0,\infty)$ implies that $f\in C^{r}(0,\infty)$,for all $r\leq l$,  ${\rm d}^{r}f(x)/{\rm d}x^{r}=\sum_{n\geq 0}{\rm d}^{r}g^{*n}(x)/{\rm d}x^{r}$ and each of these series is uniformly and absolutely convergent on any interval $(\epsilon,a)$. Next take an arbitrary interval $(2\epsilon,a)$, and choose $k$ using \eqref{Estimate1} such that $\lim_{x\to0}g^{*k\prime}(x)=0$ on $(0,a)$, and write, for any $x\in (2\epsilon,a)$, 
\begin{eqnarray}
\lefteqn{\frac{{\rm d}^{l+1}}{{\rm d}x^{l+1}}g^{*(n+2k)}(x)}&&\notag\\
&=&\frac{{\rm d}^{l}}{{\rm d}x^{l}}\int_{0}^{x}g^{*n\prime}(x-y)g^{*(2k)\prime}(y){\rm d}y\notag\\
&=&\int_{0}^{\epsilon}\frac{{\rm d}^{l+1}}{{\rm d}x^{l+1}}g^{*n}(x-y)g^{*(2k)\prime}(y){\rm d}y+\int_{0}^{x-\epsilon}g^{*n\prime}(y)\frac{{\rm d}^{l+1}}{{\rm d}x^{l+1}}g^{*(2k)}(x-y){\rm d}y\notag\\
&&+\sum_{m=1}^{l}\frac{{\rm d}^{m}}{{\rm d}x^{m}}g^{*n}(\epsilon)\frac{{\rm d}^{l+1-m}}{{\rm d}x^{l+1-m}}g^{*(2k)}(x-\epsilon)\notag\\
&=&\int_{0}^{\epsilon}\frac{{\rm d}^{l}}{{\rm d}x^{l}}g^{*n}(x-y)g^{*(2k)\prime\prime}(y){\rm d}y+\int_{0}^{x-\epsilon}g^{*n\prime}(y)\frac{{\rm d}^{l+1}}{{\rm d}x^{l+1}}g^{*(2k)}(x-y){\rm d}y\notag\\
&&+\sum_{m=2}^{l}\frac{{\rm d}^{m}}{{\rm d}x^{m}}g^{*n}(x-\epsilon)\frac{{\rm d}^{l+1-m}}{{\rm d}x^{l+1-m}}g^{*(2k)}(\epsilon).
\label{sum}
\end{eqnarray}
where we have temporarily assumed that $\lim_{x\to 0}g^{*(2k)\prime\prime}(x)=0$ in order for the third equality to be valid. Therefore we assume the latter and we will show it later. Next, we sum the left and right hand side of (\ref{sum}) with a view to establishing the induction hypothesis, i.e. that  
\begin{equation}
\sum_{n\geq 1}\frac{{\rm d}^{l+1}}{{\rm d}x^{l+1}}g^{*n}(x)=\sum_{j\leq 2k}\frac{{\rm d}^{l+1}}{{\rm d}x^{l+1}}g^{*n}(x)+\sum_{n\geq 1}\frac{{\rm d}^{l+1}}{{\rm d}x^{l+1}}g^{*(n+2k)}(x)
\label{twobits}
\end{equation} 
is uniformly and absolutely convergent on any interval $(\epsilon,a)$, which allows us to interchange differentiation and summation.

 Note that the derivatives in the finite sum on the right hand side of (\ref{twobits})  are well defined since one may follow a similar procedure to the justification given above for the continuous $(l+1)$-th order  derivatives of $g^{*(n+2k)}(x)$. Indeed in the aforementioned argument,  the role of $g^{*2k}$ is played instead by $g'$. Uniform and absolute convergence on any interval $(\epsilon, a)$ for the finite sum in (\ref{twobits}) thus follows.  

 Next we turn to the second sum on the right hand side of (\ref{twobits}). Note that the first term on the right hand side of (\ref{sum}) is absolutely and uniformly summable with respect to $n$ on $x\in (2\epsilon,a)$ due to the induction hypothesis, i.e. $g\in C^{l}(0,\infty)$ implies that $f\in C^{l}(0,\infty)$, ${\rm d}^{l}f(x)/{\rm d}x^{l}=\sum_{n\geq 0}{\rm d}^{l}(g^{*n}(x))/{\rm d}x^{l}$ and the series is uniformly and absolutely convergent on any interval $(\epsilon,a)$. Also, for the same reason, the third term on the right hand side of (\ref{sum}) is uniformly and absolutely summable with respect to $n$ on $(2\epsilon,a)$ and defines a continuous function on $(2\epsilon,a)$. Finally note that $f'(x)=\sum_{n\geq 1}g^{*n\prime}(x)$ is absolutely integrable at zero. The latter is a consequence of the fact that
\begin{equation}\label{NEW} g^{*(n+2k)\prime}(x)=\int_{0}^{x}g^{*n}(x-y)g^{*(2k)\prime\prime}(y){\rm d}y,\end{equation}
the fact that
\begin{eqnarray*}
\sum_{n\geq 0} |g^{*(n+2k)\prime}(x)|&\leq& \int_{0}^{x}\sum_{n\geq 1}|g^{*n}(x-y)||g^{*(2k)\prime\prime}(y)|{\rm d}y\\
&\leq& \sum_{n\geq 1}|g^{*n}(x)|\sup_{y\leq x}|g^{*(2k)\prime\prime}(y)|x
\end{eqnarray*}
and the fact that $\sum_{n\geq 1}|g^{*n}(x)|$ is uniformly convegent on $(0,a)$, and $\lim_{x\to 0}g^{*(2k)\prime\prime}(x)=0$. We are thus able to conclude that the series
\[\sum_{n\geq 1}\int_{0}^{x-\epsilon}g^{*n\prime}(y)\frac{{\rm d}^{l+1}}{{\rm d}x^{l+1}}g^{*(2k)}(x-y){\rm d}y,\]
and hence the sum with respect to $n$ of the second term on the right hand side of (\ref{sum}),
is uniformly and absolutely convergent on $(2\epsilon,a)$.
Thus up to showing that $\lim_{x\to 0}g^{*(2k)\prime\prime}(x)=0$ , we have completed the proof of the claim that $g\in C^{l+1}(0,\infty)$ implies that $f\in C^{l+1}(0,\infty)$.

We thus turn to showing that $\lim_{x\to 0}g^{*(2k)\prime\prime}(x)=0$. Write
\[g^{*2\prime\prime}(x)=2\big(g'\big(\frac{x}{2}\big)\big)^{2}+2\int_{0}^{\frac{x}{2}}g''(x-y)g'(y){\rm d}y\]
and use from the assumption of the theorem that $g''(x)\leq C(a)/x^{\alpha+1}$ and $|g'(x)|\leq C(a)/x^{\alpha}$ on $(0,a)$ to estimate, for each $x\in (0,a)$,
\[|g^{*2\prime\prime}(x)|\leq \frac{2C^{2}(a)}{x^{2\alpha}}+\frac{2C^{2}(a)}{x^{2\alpha}}\int_{0}^{\frac{1}{2}}\frac{1}{(1-v)^{1+\alpha}}\frac{1}{v^{\alpha}}dv\leq C_{2}(a)x^{-2\alpha}.\] 
Using this and \eqref{Estimate1}, it is trivial to show on $(0,a)$ that 
\begin{equation}\label{NEW1} |g^{*(k)\prime\prime}(x)|\leq C_{k}(a)x^{k-2-k\alpha}.\end{equation}
By choosing $k$ such that $k-1-k\alpha>1/2$ we have $\lim_{x\to 0}g^{*(2k)\prime\prime}(x)=0$. 

Since, in the above reasoning, $\epsilon$ may be taken arbitrarily small, we may finally claim the validity of the the induction hypothesis at the next iteration; namely that $g \in C^{l+2}(0,\infty)$ implies that $f \in C^{l+2}(0,\infty)$.

\bigskip

For the converse of the latter conclusion, assume that $f \in C^{l+2}(0,\infty)$. We know that $g$ is at least $C^{2}(0,\infty)$ thanks to the conditions of the theorem and therefore from the proof so far we see that $f\in C^{2}(0,\infty)$. Next we proceed to check that $|f'(x)|\leq D(a)/x^{\alpha}$ and $|f''(x)|\leq D(a)/x^{\alpha+1}$ on any interval $(0,a)$. The first is follows from \eqref{NEW} together with an application of \eqref{Estimate1} since 
\[f'(x)=\sum_{j\leq 2k}g^{*j\prime}(x)+\sum_{j> 2k}g^{*j\prime}(x)=\sum_{j\leq 2k}g^{*j\prime}(x)+\int_{0}^{x}f(x-y)g^{*(2k)}(y){\rm d}y.\]
We also have that $|f''(x)|\leq D(a)/x^{\alpha+1}$ is a consequence of $|f'(x-y)|\leq D(a)/(x-y)^{\alpha}$,
\[f'(x)=\sum_{j\leq 2k}g^{*j\prime}(x)+\sum_{j> 2k}g^{*j\prime}(x)=\sum_{j\leq 2k}g^{*j\prime}(x)+\int_{0}^{x}f'(x-y)g^{*(2k)}(y){\rm d}y\]
and \eqref{NEW1}.

Now turning to the equation
$f=1+f*g$ we obtain by rearrangement $g=(f-1)+g*(1-f)$ and therefore
\[g=f-1+\sum_{n\geq 1}(f-1)*(1-f)^{*n}.\]
Hence the arguments used for studying smoothness of $f$ given smoothness of $g$, via the equation $f=1+f*g$, are applicable to the series above and we can conclude that $f \in C^{l+2}(0,\infty)$ implies that $g \in C^{l+2}(0,\infty)$. 
\end{proof}

\section{General remark on scale function proofs}\label{generalremark}

Before we proceed to the proofs of Theorems \ref{I}, \ref{II} and \ref{III}, we make a general remark which applies to all three of the latter. Specifically we note that it suffices to prove these three theorems for the case that $q=0$ and  $\psi'(0+)\geq 0$. In the case that $\psi'(0+)<0$ and/or $q\geq 0$, we know from (\ref{reduce}) that smoothness properties of $W^{(q)}$ reduces to smoothness properties of $W_{\Phi(q)}$. Indeed, on account of the fact that $\Phi(q)>0$ (and specifically $\Phi(0)>0$ when $\psi'(0+)<0$ and $q=0$), it follows that 
$\psi_{\Phi(q)}'(0+)=\psi'(\Phi(q))>0$. The smoothness properties of $W_{\Phi(q)}$ would then be covered by the proofs of the smoothness properties of $W$ under the assumption that $\psi(0+)\geq 0$ providing that the L\'evy measure $\Pi_{\Phi(q)}({\rm d}x) = e^{\Phi(q)x}\Pi({\rm d}x)$ simultaneously respects the conditions of the Theorems \ref{I}, \ref{II} and \ref{III}. However, it is clear that this is the case as it is the behaviour of $\Pi_{\Phi(q)}$ in the neighbourhood of the origin which matters.

\section{Proof of Theorem \ref{I}}

We give two proofs of Theorem \ref{I}. The first proof appeals to the just established conclusions regarding renewal measures, taking advantage of the remarks made around (\ref{remark-around}). The second proof gives a more probabilistic approach which takes advantage of the natural connection between scale functions and  the excursion measure $n$. Recall from the previous section that, without loss of generality, we may assume throughout 
that $\psi'(0+)\geq 0$.

\noindent\begin{proof}[First proof] 
Assume that $X$ oscillates (equivalently $\psi'(0+)=0$). Then from \eqref{remark-around}  we have that $W(x)$ is the potential measure of the descending ladder height process. 
Then \eqref{Crossing} holds with $u(x) = W'(x)$, $\delta=\sigma^{2}/2$ and $g(x)=-\int_{0}^{x}\overline{\mu}(y){\rm d}y=-\int_{0}^{x}\overline{\overline{\Pi}}(y){\rm d}y$ where $$\overline{\overline{\Pi}}(x): = \int_x^\infty \overline{\Pi}(y){\rm d}y.$$
On the other hand when $X$ drifts to infinity, that is to say $\psi'(0+)>0$, we may still identify $W$ as the renewal function of the descending ladder height process and moreover, from the Wiener-Hopf factorisation one also easily deduces that 
$W(x)=P(\inf_{s\geq0}X_{s}\geq-x)/\psi'(0+)$, see for example Chapter 8 of Kyprianou (2006). The probability that the descending ladder height subordintor crosses level $x>0$ is now equal to $1- P(\inf_{s\geq0}X_{s}\geq-x)$. Therefore \eqref{Crossing} is slightly transformed to
\begin{equation}\label{Drift}
\frac{\sigma^{2}}{2}W'(x)+\int_{0}^{x}W'(x-y)\overline{\overline{\Pi}}(y){\rm d}y=1-\psi'(0+)\int_{0}^{x}W'(y){\rm d}y.
\end{equation}
Putting $f(x)=\sigma^{2}W'(x)/2$ and
$g(x)=-2 \sigma^{-2}\int_{0}^{x}\overline{\overline{\Pi}}(y){\rm d}y-2\psi'(0+)\sigma^{-2}x$
in \eqref{Drift} we get 
$f=1+f*g$.
Note that the latter also agrees with the case $\psi'(0+)=0$.

In either of the two cases (oscillating or drifting to infinity) we may now  apply Theorem \ref{Sub} in a similar way to the way it was used in the proof of Corollary \ref{IV} to deduce the required result.
\end{proof}

\bigskip

\begin{proof}[Second proof]
Recall from the introduction that $W^{\prime}(x)=n(\overline\epsilon \geq
x )W(x)$ and hence if the limits exist, then
\begin{eqnarray}
W^{\prime\prime}(x+) &=& \lim_{\varepsilon\downarrow 0}\frac{%
W^{\prime}(x+\varepsilon) - W^{\prime}(x)}{\varepsilon}  \notag \\
&=&-\lim_{\varepsilon\downarrow 0} \frac{ n(\overline\epsilon \in
[x,x+\varepsilon) ) W(x )- n(\overline\epsilon \geq x + \varepsilon) (W(x +
\varepsilon)- W(x))}{\varepsilon}  \notag \\
&=&-\lim_{\varepsilon\downarrow 0}\frac{n(\overline\epsilon \in
[x,x+\varepsilon) )}{\varepsilon} W(x) + n(\overline\epsilon \geq
x)W^{\prime}(x+ ).  \label{first step}
\end{eqnarray}

To show that the limit on the right hand side of (\ref{first step}) exists,
define $\sigma_x = \inf\{t>0 : \epsilon_t \geq x\}$ and $\mathcal{G}_t =
\sigma(\epsilon_s : s\leq t)$. With the help of the Strong Markov Property
for the excursion process we may write
\begin{eqnarray}
n(\overline\epsilon \in [x,x+\varepsilon) )&=&n(\overline\epsilon \geq x,
\epsilon_{\sigma_x}< x+\varepsilon, \overline\epsilon < x+\varepsilon) 
\notag \\
&=&n( \mathbf{1}_{(\overline\epsilon \geq x, \epsilon_{\sigma_x}<
x+\varepsilon)}n(\overline\epsilon < x+\varepsilon| \mathcal{G}_{\sigma_x}))
\notag \\
&=&n( \mathbf{1}_{(\overline\epsilon \geq x, \epsilon_{\sigma_x}<
x+\varepsilon)}P_{-\epsilon_{\sigma_x}}(\tau^+_0 <
\tau^-_{-(x+\varepsilon)}))  \notag \\
&=& n(\overline\epsilon \geq x, \epsilon_{\sigma_x}= x)\frac{W(\varepsilon)}{%
W(x+\varepsilon)}  \notag \\
&& + n( \mathbf{1}_{(\overline\epsilon \geq x,x< \epsilon_{\sigma_x}<
x+\varepsilon)}\frac{W(x+\varepsilon - \epsilon_{\sigma_x})}{W(x+\varepsilon)%
}).  \label{see creep}
\end{eqnarray}

A L\'evy process may creep both upwards
and downwards if and only if it has a Gaussian component (see Bertoin (1996)
pp175). Since a spectrally negative L\'evy process always creeps upwards
then we know that it creeps downwards if and only if it has a Gaussian
component. Hence we may say that for the case at hand $%
\{\epsilon_{\sigma_x}=x\}$ has non-zero $n$-measure.

Next note that the  parameter $\sigma>0$ is identified via the limit 
\begin{equation}
\lim_{\theta\uparrow\infty}\frac{\psi(\theta)}{\theta^2} = \sigma^2 /2.
\label{s/2}
\end{equation}
It is also known from Bertoin (1996) that $W(0+)=0$ as a consequence of the fact
that the lower half line is regular for 0 for $X$.






From (\ref{LapT}) and the known fact that in this case $W(0+)=0$ 
\begin{equation*}
\int_0^\infty e^{-\theta x}W^{\prime}(x)\mathrm{d}x =\theta\int_0^\infty
e^{-\theta x}W(x)\mathrm{d}x={\frac{\theta}{\psi(\theta)}}.
\end{equation*}
Hence using (\ref{s/2}) 
\begin{equation}  \label{W-dash0}
W^{\prime}(0+)=\lim_{\theta\to\infty}\theta \int_0^\infty e^{-\theta
x}W^{\prime}(x)\mathrm{d}x =\lim_{\theta\to\infty}{\frac{\theta^2}{%
\psi(\theta)}}={\frac{2}{\sigma^2}}.
\end{equation}















Using the last equality together with the monotonicity of $W$ we have that
\begin{eqnarray}
\lefteqn{\limsup_{\varepsilon\downarrow 0}\frac{1}{\varepsilon}n( \mathbf{1}%
_{(\overline\epsilon \geq x, \epsilon_{\sigma_x}\in(x, x+\varepsilon))}\frac{%
W(x+\varepsilon - \epsilon_{\sigma_x})}{W(x+\varepsilon)})}  \notag \\
&\leq& \frac{1}{W(x)} \limsup_{\varepsilon\downarrow 0}n(\overline\epsilon
\geq x, \epsilon_{\sigma_x}\in(x,x+\varepsilon))\frac{W(\varepsilon)}{%
\varepsilon}  \label{non-creep} \\
& =&0.  \notag
\end{eqnarray}

In conclusion, $W^{\prime\prime}(x+)$ exists and 
\begin{equation*}
W^{\prime\prime}(x+)=-W^{\prime}(0+)n(\overline\epsilon \geq x,
\epsilon_{\sigma_x}= x)+n(\overline\epsilon \geq x)W^{\prime}(x),
\end{equation*}
that is to say,
\begin{equation}
n(\overline\epsilon \geq x,
\epsilon_{\sigma_x}= x)=\frac{\sigma ^{2}}{2}\left\{\frac{W^{\prime
 }(x)^{2}}{W(x)}- W^{\prime \prime }(x+)\right\}
 \label{W-doubledash}
\end{equation}

 We shall now show that there exists a left second derivative $W^{\prime
 \prime }(x-)$ which is also equal to the right hand side of (\ref{W-doubledash}) thus
 completing the proof. We recall from Doney (2005) 
 that $A\in \mathcal{F}_{t}$
\begin{equation*}
n(A,t<\zeta )=\lim_{y\downarrow 0}\frac{\widehat{ P}_{y}(A,t<\tau
_{0}^{-})}{y}.
\end{equation*}
From this we may write 
 \begin{equation}
 n(\overline{\epsilon }\geq x,\epsilon _{\sigma _{x}}=x)=\lim_{y\downarrow 0}%
  \frac{\widehat{ P}_{y}(X_{\tau _{x}^{+}}=x;\tau _{x}^{+}<\tau _{0}^{-})}{y}%
   \label{radlimit}
 \end{equation}%
 where $\widehat{ P}_y$ is the law of $-X$ when issued from $y$.
 From general potential theory (cf. Kesten (1969)) we also know that 
 \begin{equation*}
 \widehat{ P}_{y}(X_{\tau _{x}^{+}}=x;\tau _{x}^{+}<\infty )=\frac{\sigma
 ^{2}}{2} W^{\prime }(x-y).
 \end{equation*}%
 See also the review paper of Pistorius (2005) for a recent proof of this
 fact. Using the Strong Markov Property, the above formula and the fact that $X$ creeps upwards it is straightforward to deduce that 
\[
 \widehat{ P}_{y}(X_{\tau _{x}^{+}}=x;\tau _{0}^{-}< \tau _{x}^{+})
=\frac{W(x-y)}{W(x)} \times \frac{\sigma ^{2}}{2}W^{\prime }(x)
\]
and hence
 \[
 \widehat{ P}_{y}(X_{\tau _{x}^{+}}=x;\tau _{x}^{+}<\tau _{0}^{-}) 
 =\frac{\sigma ^{2}}{2}\left\{ W^{\prime }(x-y)-\frac{W(x-y)}{W(x)}%
 W^{\prime }(x)\right\} .
 \]%
 Returning to (\ref{radlimit}) we compute 
 \begin{eqnarray*}
 n(\overline{\epsilon }\geq x,\epsilon _{\sigma _{x}}=x) &=&\lim_{y\downarrow
 0}-\frac{\sigma ^{2}}{2}\left\{ \frac{W^{\prime }(x)-W^{\prime }(x-y)}{y}+%
 \frac{W^{\prime }(x)}{W(x)}\frac{W(x)-W(x-y)}{y}\right\}  \\
 &=&\frac{\sigma ^{2}}{2}\left\{\frac{W^{\prime
 }(x)^{2}}{W(x)}- W^{\prime \prime }(x-)\right\} 
 \end{eqnarray*}%
 Note that the existence of $W^{\prime \prime }(x-)$ is guarenteed in light
 of (\ref{radlimit}). Taking account of (\ref{W-dash}) and (\ref{W-dash0}) we see
 that the final equality above is identical to (\ref{W-doubledash}) but with $%
 W^{\prime \prime }(x+)$ replaced by $W^{\prime \prime }(x-)$.

 Thus far we have shown that a second derivative exists everywhere. To complete the proof, we need to show that this second derivative is continuous. To do this, it suffices to show that $n(\overline{\epsilon }\geq x,\epsilon _{\sigma _{x}}=x)$ is continuous. To this end note that
\begin{eqnarray*}
 n(\overline{\epsilon }\geq x,\epsilon _{\sigma _{x}}=x)  &=&n(\overline{\epsilon }\geq x)-n(\overline{\epsilon }\geq x,\epsilon _{\sigma _{x}}>x)\\
&=& \frac{W'(x)}{W(x)} - \int_0^x \left\{W'(x-y) - \frac{W'(x)}{W(x)}W(x-y)\right\}\overline{\Pi}(y){\rm d}y
\end{eqnarray*}
where the second term on the right hand side is taken from Lemma 2.1 of Kyprianou and Zhou (2009). The required continuity is now immediate from the right hand side above.
\end{proof}





\section{Proof of Theorems \ref{II}, \ref{III}} 
For both proofs, following the remarks in Section \ref{generalremark} we again assume without loss of generality that $\psi'(0+)\geq 0$.

\begin{proof}[Proof of Theorem \ref{II}]
Recall from the proof of Theorem \ref{I} the identity (\ref{Drift}) and the choices $f(x)=\sigma^{2}W'(x)/2$ and
$g(x)=-2 \sigma^{-2}\int_{0}^{x}\overline{\overline{\Pi}}(y){\rm d}y-2\psi'(0+)\sigma^{-2}x$ which transform it to the renewal equation $f= 1+ f*g.$

 Our objective is to recover the required result from Theorem \ref{Renewal}. In order to do this, we need to show that $|g'(x)|\leq C(a)x^{-\alpha}$ and $|g''(x)|C(a)x^{-\alpha-1}$ for some $\alpha<1$ on any interval $(0,a)$.  
The assumption on the Blumenthal-Getoor index which implies that there exists a $\vartheta\in(1,2)$ such that
$\int_{-1}^{0}|x|^{\vartheta}\Pi(dx)<\infty.$
Therefore
\[|\epsilon|^{\vartheta}\big(\overline{\Pi}(\epsilon)-\overline{\Pi}(1)\big)\leq\int_{-1}^{-\epsilon}|x|^{\vartheta}\Pi(dx)<\infty\]
and it follows that $|\epsilon|^{\vartheta}\overline{\Pi}(\epsilon)\leq C$, for each $\epsilon\leq 1$. This and the fact that $\lim_{x\to \infty}\overline{\Pi}(x)=0$ show that $|g''(x)|\leq C(a)x^{-\vartheta}$, for each interval $(0,a)$. 

Next, from the definition of $\overline{\overline{\Pi}}(x)$ and the bound just obtained for $\overline{\Pi}(x)$, we  get for $x\leq 1$,
\[
\overline{\overline{\Pi}}(x)\leq \int_{x}^{1}\frac{C(1)}{y^{\vartheta}}{\rm d}y+\overline{\Pi}(1)\leq \frac{D(1)}{x^{\vartheta-1}}
\]
for an approprite constant $D(1)>0$.
Since $\lim_{x\to \infty}\overline{\overline{\Pi}}(x)=0$ we conclude that $\overline{\overline{\Pi}}(x)\leq C(a)x^{1-\vartheta}$ on every interval $(0,a)$ where $C(a)>0$ plays the role of a generic appropriate constant.
Hence Theorem \ref{Renewal} is applicable and the proof is complete.
\end{proof}

\bigskip

\begin{proof}[Proof of Theorem \ref{III}] Suppose as usual that $\psi'(0+)\geq 0$.
In this case it is well known that the Laplace exponent can  be rewritten in the form 
\begin{eqnarray}
\psi(\theta) 
&=& \mathrm{c}\theta - \theta \int_0^\infty e^{-\theta x}\overline\Pi(x)%
\mathrm{d}x \text{ for }\theta\geq 0  \label{bvLE}
\end{eqnarray}
where necessarily $\mathrm{c}>0$ and $\int_{(-\infty,0)}(1\wedge
|x|)\Pi(\d x)<\infty$; see Chapter VII of Bertoin (1996). On account of the fact that $\psi'(0+)\geq 0$,  it follows  from (\ref{bvLE})  that
\begin{equation*}
\mathrm{c}\geq \int_0^\infty \overline\Pi(x)\mathrm{d}x>   \int_0^\infty e^{-\theta x} \overline\Pi(x)\mathrm{d}x 
\end{equation*}
for all $\theta>0$.
Using the above inequality to justify convergence, we may write  
\begin{equation}
\frac{\theta}{\psi(\theta)} =\frac{1}{\mathrm{c}}\left( \frac{1}{1- \frac{1}{%
\mathrm{c}}\int_0^\infty  e^{-\theta x}\overline\Pi(x)\mathrm{d}x} \right)
= \frac{1}{\mathrm{c}}\sum_{n\geq 0} \frac{1}{\mathrm{c}^n}\left(
\int_0^\infty  e^{-\theta x}\overline\Pi(x)\mathrm{d}x \right)^n.
\label{series}
\end{equation}
From (\ref{remark-around}) 
we deduce that 
\begin{equation}
W(x) = \frac{1}{\mathrm{c}}\sum_{n\geq 0}g^{*n}(x), 
\label{spectral}
\end{equation}
where
\begin{equation*}
g( x) =\frac{1}{\mathrm{c}} \int_0^x\overline\Pi(x)\mathrm{d}x. 
\end{equation*}
Taking $f(x) = {\rm c}W(x)$ we see that we are again reduced to studying the equation $f = 1 + f*g$.   The proof now follows as a direct consequence of Theorem \ref{Renewal} with the conditions on $g$ following as a straightforward consequence of the fact  that $g''(x) = {\rm c}^{-1}\pi(x)$ together with the assumption on the latter density.
\end{proof}

\section*{Acknowledgments} 
We gratefully appreciate extensive discussions with Renming Song and Ron Doney, both of whom brought a number of important mathematical points to our attention. An earlier version of this paper has been cited as Chan and Kyprianou (2006). AEK acknowledges the support of EPSRC grant number EP/E047025/1.

\end{document}